\documentclass[11pt, letterpaper]{article}

\pdfoutput=1

\usepackage{amsmath}
\usepackage{amsfonts}
\usepackage{amssymb} 
\usepackage{amsthm}
\usepackage{mathrsfs} 
\usepackage[nottoc,numbib]{tocbibind} 

\usepackage[pdftitle={When are radicals of Lie groups lattice-hereditary?},pdfauthor={Andrew Geng}]{hyperref}
\usepackage{url}

\title{When are radicals of Lie groups lattice-hereditary?}
\author{Andrew Geng}

\theoremstyle{plain}
\newtheorem{thm}{Theorem}[section]
\newtheorem{claim}[thm]{Claim}
\newtheorem{lemma}[thm]{Lemma}
\newtheorem{cor}[thm]{Corollary}
\theoremstyle{definition}
\newtheorem{defn}[thm]{Definition}
\newtheorem{eg}[thm]{Example}
\newtheorem{qq}[thm]{Question}
\theoremstyle{remark}
\newtheorem{rmk}[thm]{Remark}
\newtheorem{notation}[thm]{Notation}

\newcommand{\keyword}{\emph}
\newcommand{\op}{\operatorname}
\newcommand{\R}{\mathbb{R}}
\newcommand{\Z}{\mathbb{Z}}
\newcommand{\lie}{\mathfrak}

\begin{document}

\maketitle

\begin{abstract}
    This note aims to clarify
    what conditions on a connected Lie group $G$
    imply that its maximal connected normal solvable subgroup $R$
    intersects each lattice of $G$ as a lattice in $R$.
\end{abstract}

\section{Introduction}

\subsection{Motivation}

The purpose of this note is to clarify the situation about a fundamental
claim in the general study of lattices in Lie groups. The setup is as follows.

Recall that every connected Lie group $G$ is an extension
    \[ 1 \to R \to G \to \overline{S} \to 1, \]
with $R$ solvable and $\overline{S}$ semisimple. The subgroup
$R$, called the \keyword{radical} of $G$,
is the unique maximal connected, normal, and solvable subgroup of $G$.
A subgroup $S \subseteq G$, called the \keyword{semi\-simple part},
covers $\overline{S}$ via the map $G \to \overline{S}$
(not necessarily finitely) and is unique up to conjugacy.
This divides much of the study of a general Lie group into the
study of $R$ and $S$.

A \keyword{lattice} in $G$ is
a discrete subgroup $\Gamma$ for which $\Gamma\backslash G$ has finite measure
(induced by Haar measure on $G$). Attempting to achieve
the above division for lattices, one can ask the following.
\begin{qq}
    If $\Gamma$ is a lattice in $G$, is $\Gamma \cap R$ a lattice in $R$?
\end{qq}
When it is, one then obtains
    \( 1 \to \Gamma \cap R \to \Gamma \to \Gamma / (\Gamma \cap R) \to 1, \)
and it is known (see Thm.~\ref{heredity_and_closure} below)
that $\Gamma/(\Gamma \cap R)$ is then a lattice in $\overline{S}$.

One can also ask whether $\Gamma \cap N$ is a lattice in $N$
where $N$ is the unique maximal connected, normal, and \emph{nilpotent}
subgroup of $G$ (the \emph{nilradical}).
So define a Lie subgroup $H$ of $G$ to be \keyword{lattice-hereditary}
if $\Gamma \cap H$ is a lattice in $H$ for each lattice $\Gamma$ of $G$.
Raghunathan has made the following positive claim.
\begin{claim}[{\cite[Corollary 8.28]{raghunathan}}] \label{main_broken}
    If $G$ is connected and no compact factor
    of $S$ acts trivially on $R$,
    then $N$ and $R$ are both lattice-hereditary.
\end{claim}
In \cite{starkov}, Starkov produced the counterexample
    \[ G = \left(\R^2 \rtimes \op{SO}(1,1)^0\right) \times \left(\R^3 \rtimes \op{SO}(3)\right) \]
(details below in Example \ref{eg:starkov}).
In response, Wu gave a revised proof in \cite[Proposition 1.3]{wu},
following Mostow's proof in \cite[Lemma 3.9]{mostow} that $N$ is hereditary.
An internet search for citations
indicates awareness of these papers by later authors
but scant further elaboration, found mostly in \cite[p.~107]{onishchik2}.
Except for Wu's comments in \cite[\S{}2]{wu} about why
one step in Raghunathan's proof is false,
no one points out specific mistakes in previous arguments.
Some authors have explicitly chosen to ``refrain from taking sides
in the discussion'' \cite[Rem.~6]{kammeyer_farrell-jones_2014}.

In his review \cite{hu88} of Wu's proof, Humphreys
encourages the reader to ``study these arguments independently,
since they involve a complicated mixture of techniques.''
Misprints (noted in the review) and sentence fragments
further complicate reading, and Wu's claim ultimately turns out
to be incorrect. Since the literature does not contain
a correction of Claim \ref{main_broken}
that accounts for what Wu's method can achieve,
the author wishes to give one (Theorem \ref{main}, below)
and reconcile it with other results.

\subsection{Results}

Applying Wu's revised proof
step-by-step to Starkov's counterexample
reveals that a step elided in Raghunathan's proof and made explicit in Wu's
is missing an assumption in both versions.
Adding it yields the following.
\begin{thm}[Revised Claim \ref{main_broken}] \label{main}
    Let $G$ be a connected Lie group whose semisimple part $S$ has no
    compact factor acting trivially on the radical $R$ of $G$. Then
    \begin{itemize}
        \item[(i)] The nilradical $N$ of $G$ is lattice-hereditary.
        \item[(ii)] If no compact factor of $S$ acts trivially on $R/N$,
            then $R$ is also lattice-hereditary.
    \end{itemize}
\end{thm}
Mostow proved Part (i) in \cite[Lemma 3.9]{mostow}.
The new assumption, in (ii), is required by
Starkov's example (see Remark \ref{rmk:what_is_broken} below).
In fact, via a theorem of Chevalley, (ii) is a case of
\cite[Thm.~1]{auslander_radicals_1963} by Auslander
(see Remark \ref{rmk:chevalley}).
So what Wu's method does is to unify these theorems of
Mostow and Auslander, which we express for all connected $G$ as follows.
(See Section \ref{sec:proof_cor} for the proof.)
\begin{cor} \label{cor:main}
    Let $G$ be a connected Lie group with radical $R$, nilradical $N$, and semisimple part $S$.
    Let $C$ and $S_K$ be the maximal connected semisimple compact normal subgroups
    of $G$ and $S$, respectively. The following are lattice-hereditary in $G$.
        \[ C \subseteq NC \subseteq NS_K \subseteq RS_K \]
\end{cor}
\begin{rmk} \label{rmk:auslander}
    Part (i) of Theorem \ref{main} is recovered by assuming $C = \{1\}$,
    and Part (ii) by assuming $S_K = \{1\}$.
    Auslander's result, as originally stated, is the heredity of $RS_K$.\footnote{
        The version in
        \cite[Thm.~1]{auslander_radicals_1963}
        requires uniform (i.e. cocompact) lattices.

        For the non-uniform case (which is claimed in
        e.g.\ \cite[I.4 Thm.~1.7]{onishchik3}),
        one could replace \cite[Prop.~3]{auslander_radicals_1963}
        with \cite[Lemma 3.4(d)]{mostow} to show $\Gamma R K$
        is closed with identity component $RK$ and use
        \cite[Lemma 2.5]{mostow_homogeneous_1962} to finish.
        We will instead use these ingredients in Theorem \ref{main}
        and use that to prove this part of Corollary \ref{cor:main}.
    }
\end{rmk}

After definitions and basic facts in Section \ref{sec:bg},
we recall examples (including Starkov's) in Section \ref{sec:examples}
and the proof in Section \ref{sec:proof}. This includes the correction,
an alternative proof of the key Lemma \cite[Lemma 3.8]{mostow},
a remark on where Starkov's example fits in,
and the proof of Corollary \ref{cor:main}.
We mention related statements in Section \ref{sec:others}.

\paragraph{Acknowledgements.}
The author wishes to thank Benson Farb for helpful discussions and extensive
comments during the preparation of this note. Thanks are also due to Daniel
Studenmund for a careful proofreading, to Dave Witte Morris for comments
on a draft, and to an anonymous referee whose review led to
a couple of corrections.

\section{Background} \label{sec:bg}

This section contains definitions of concepts mentioned
in the introduction and used in the sequel
(e.g.\ radical, nilradical, and heredity),
along with some facts used in the examples and proof.
Most of this material can be found in the books \cite{raghunathan},
\cite{onishchik2}, and \cite{onishchik3}.

\subsection{Relevant subgroups of Lie groups}

\begin{defn}[Levi decomposition, see e.g.\ {\cite[\S{}1.4]{onishchik3}}]
    \label{defn:levi}
    If $G$ is a connected Lie group, then $G = RS$ where $R$
    is the unique maximal connected normal solvable subgroup of $G$
    and $S$ is a semisimple virtual Lie subgroup covering $G/R$.
    The group $R$ is called the \keyword{radical}.
    The group $S$ is called the \keyword{semisimple part}
    and is unique up to conjugacy \cite[Thm.~1.4.3]{onishchik3}.
\end{defn}
\begin{defn}[Nilradical, see e.g.\ {\cite[\S{}2.5]{onishchik3}}]
    \label{defn:nilrad}
    If $G$ is a Lie group, its \keyword{nilradical} $N$
    is its unique maximal connected normal nilpotent subgroup.
    It coincides with the nilradical of $R$.
\end{defn}

The nilradical may not be part of an analogous decomposition---i.e.\ $G \to G/N$
may not restrict to a covering homomorphism on any virtual Lie subgroup\footnote{
    For example, if $G = \R^4 \rtimes H$
    where $H$ is the Heisenberg group acting through the composition
    $H \to H/Z(H) \cong \R^2 \to \op{SO}(2) \times \op{SO}(2)$,
    then $G$ has nilradical $N = \R^4 \times Z(H)$ and
    $G \to G/N \cong \R^2$ has no homomorphic section.
    However, see e.g.\ \cite[\S{}1.6.4]{onishchik3} for
    a related decomposition.
}---but it has the following useful relationship with $R$,
due to Chevalley.
\begin{thm}[see e.g.\ {\cite[II.7 Thm.~13]{jacobson}}] \label{thm:chevalley}
    Let $\lie{n}$, $\lie{r}$, and $\lie{g}$ be the Lie algebras of $N$, $R$, and $G$.
    Then $[\lie{g}, \lie{r}] \subseteq \lie{n}$.
\end{thm}
\begin{cor}
    $R/N$ and $\lie{r}/\lie{n}$ are abelian, and $G$
    acts trivially on them.
\end{cor}

\subsection{Heredity}

This section defines heredity
and recalls some related properties and theorems.

\begin{defn}[Heredity, following {\cite[\S{}I.1.4.2]{onishchik2}}]
    In a Lie group $G$, let
    $\Gamma$ be a lattice, and let
    $H$ be a closed (i.e.\ Lie) subgroup of $G$.
    \begin{itemize}
        \item $H$ is \keyword{$\Gamma$-hereditary}
            if $H \cap \Gamma$ is a lattice in $H$.
        \item $H$ is \keyword{lattice-hereditary} if it
            is $\Gamma$-hereditary for every lattice $\Gamma$ in $G$.
    \end{itemize}
\end{defn}

The statements of Theorem \ref{main} in \cite[8.28]{raghunathan}
and \cite[Prop.~1.3]{wu} include as conclusions some properties that
are equivalent to heredity, so we note the equivalence now.
\begin{thm}[{\cite[I.1 Thm.~4.3,5,7]{onishchik2}}] \label{heredity_and_closure}
    Let $\Gamma$ and $H$ be a lattice and a closed subgroup, respectively,
    of a Lie group $G$.
    If either $\Gamma$ is uniform (i.e.\ $\Gamma\backslash G$ is compact)
    or $H$ is normal, then the following are equivalent.
    \begin{itemize}
        \item $H\Gamma$ is a closed subset of $G$.
        \item $H$ is $\Gamma$-hereditary.
        \item The image of $\Gamma$ in $G/H$ is discrete.
        \item The image of $\Gamma$ in $G/H$ is a lattice
            (when $H$ is normal).
    \end{itemize}
\end{thm}
\begin{eg} \label{closed_but_not_hereditary}
    It can happen that
    $H\Gamma$ is closed but $H$ is not $\Gamma$-hereditary
    for non-normal $H$ and non-uniform $\Gamma$,
    e.g.\ if $G = \op{SL}(2,\R)$ and $\Gamma = \op{SL}(2,\Z)$
    with $H$ being the diagonal matrices.
    \cite[I.1 Example 4.6]{onishchik2}
\end{eg}
\begin{rmk}
    Example \ref{closed_but_not_hereditary} is not a counterexample
    to \cite[Lemma 2.5]{mostow_homogeneous_1962} (Lemma \ref{lemma:nested} below)
    because here $H\Gamma \subseteq G$ is only a subset, not a subgroup.
\end{rmk}
To study lattices in general Lie groups,
we will use two facts about lattices in solvable groups,
due to Mostow.
\begin{thm} \label{latts_in_sol}
    Let $G$ be a connected solvable Lie group.
    \begin{itemize}
        \item[(i)] Every lattice of $G$ is uniform
            \cite[Thm.~6.2]{mostow_homogeneous_1962}.
        \item[(ii)] The nilradical $N$ of $G$ is lattice-hereditary
            \cite[\S{}5]{mostow_factor_1954}.\footnote{
                Cited as ``Theorem 4.1'' from \cite{mostow_factor_1954} in
                the proof of \cite[Lemma 3.9]{mostow}.
            }
    \end{itemize}
\end{thm}

\section{Cautionary examples} \label{sec:examples}

This section contains two examples illustrating the
necessity of the hypotheses in Theorem \ref{main}.
The second example is due to Starkov in \cite{starkov}.
We give some details in order to fill the gaps mentioned
by its review in \cite{hu84}.

\subsection{A lowest-dimensional group with non-hereditary radical}

The following example establishes the necessity of the hypothesis
(included in most versions of Theorem \ref{main}) that no compact
factor of $S$ acts trivially on $R$.

\begin{eg}[see e.g.\ {\cite[Ex.~2.3]{van_limbeek_riemannian_2014}}]
    \label{eg:minimal}
    Define
    \begin{align*}
        \rho: \Z &\to \R \times \op{SO}(3) \\
            n &\mapsto (n, A^n)
    \end{align*}
    where $A \in \op{SO}(3)$ has infinite order.
    Then $\rho(\Z)$ is a lattice in $\R \times \op{SO}(3)$,
    and the radical $\R$ of $\R \times \op{SO}(3)$ is not $\rho(\Z)$-hereditary.
\end{eg}
\begin{proof}
    Let $G = \R \times \op{SO}(3)$.
    The projection $G \to \R$ takes $\rho(\Z)$ injectively
    to the discrete group $\Z$,
    so $\rho(\Z)$ is discrete in $G$. A fundamental domain of the
    action of $\rho(\Z)$ on $G$ is contained in $[0,1] \times \op{SO}(3)$,
    which has finite volume. Therefore $\rho(\Z)$ is a lattice.

    Since $A$ has infinite order, $\rho(\Z) \cap \R = 0$,
    which is not a lattice in $\R$.
\end{proof}

\subsection{Starkov's counterexample}

The following example, due to Starkov, refutes Claim \ref{main_broken},
including the version written in \cite[Prop.~1.3]{wu}.
(See Remark \ref{rmk:what_is_broken}.)

\begin{eg}[{\cite{starkov}}] \label{eg:starkov}
    Let $\op{SO}(1,1)^0$ denote the identity component of $\op{SO}(1,1)$. The radical $R$ of
        \[ G = \left(\R^2 \rtimes \op{SO}(1,1)^0\right) \times \left(\R^3 \rtimes \op{SO}(3)\right) \]
    is not lattice-hereditary, as demonstrated by the following lattice $\Gamma$.

    Choose $(s,r) \in \op{SO}(1,1)^0 \times \op{SO}(3)$
    where $s$ and $r$ act with the following
    characteristic polynomials.
    \begin{align*}
        P_s(x) &= x^2 - \big(4 + \sqrt{8}\big)x + 1 \\
        P_r(x) &= \left(x^2 - \big(4 - \sqrt{8}\big)x + 1\right)(x-1)
    \end{align*}
    The basis of $\R^2 \times \R^3$
    in which $(s,r)$ acts in Frobenius normal form
    generates a group $\Gamma_0 \cong \Z^5$.
    Let $\Gamma$ be the group generated by $(s,r)$ and $\Gamma_0$.
\end{eg}
\begin{proof}[Proof that $\Gamma$ is a (uniform) lattice]
    Since $\Gamma_0$ is generated by a basis of $\R^2 \times \R^3$,
    it is a lattice in $\R^2 \times \R^3$. Then some open
    $V \subset \R^2 \times \R^3$ meets $\Gamma_0$ in only the identity.
    Define
        \[ W = V \cdot \left(\left\{a \in \op{SO}(1,1)^0 \left|\; \op{tr} a < 4 + \sqrt{8}\right.\right\} \times \op{SO}(3)\right) . \]
    In the topology on a semidirect product, $W$ is open in $G$.
    Since $\op{tr} s = 4 + \sqrt{8}$,
    the projection of $W \cap \Gamma$ to $\op{SO}(1,1)^0$ is trivial.

    The characteristic polynomial by which $(s,r)$ acts on $\R^2 \times \R^3$ is
        \[ P_s(x) P_r(x) = (x-1)(x^4 - 8x^3 + 10x^2 - 8x + 1) . \]
    This has integer coefficients,
    so conjugation by $(s,r)$ preserves $\Gamma_0$.
    Therefore elements of $\Gamma$ with trivial
    projection to $\op{SO}(1,1)^0$ lie in $\Gamma_0$.
    Then $W \cap \Gamma = W \cap \Gamma_0 = V \cap \Gamma_0$,
    which contains only the identity; so $\Gamma$ is discrete in $G$.

    If $U$ is a closed fundamental domain for the action of $\Gamma_0$ on $\R^2 \times \R^3$,
    then a closed fundamental domain for the action of $\Gamma$ on $G$ is the set
        \[ U \cdot \left(\left\{a \in \op{SO}(1,1)^0 \left|\; \op{tr} a^2 \leq 4 + \sqrt{8}\right.\right\} \times \op{SO}(3)\right) . \]
    In the topology on a semidirect product,
    this set is diffeomorphic to the product of $\op{SO}(3)$ and a $6$-cube;
    so $\Gamma \backslash G$ is compact.
\end{proof}
\begin{proof}[Proof that the radical is not $\Gamma$-hereditary]
    Projection $G \to \op{SO}(3)$ has simple image
    and solvable kernel $(\R^2 \rtimes \op{SO}(1,1)^0) \times \R^3$,
    so $R$ is this kernel.
    The eigenvalues of $r$ have Galois conjugates off the unit circle
    (namely the eigenvalues of $s$), so none are roots of unity.
    Then $r$ has infinite order in $\op{SO}(3)$,
    so $\Gamma \cap R$ is only the trivial $(s,r)$-translate
    $\Gamma_0$.

    By dropping the $\R^2 \times \R^3$ coordinates,
    $\Gamma_0 \backslash R$ surjects onto $\op{SO}(1,1)^0 \cong \R$;
    so $\Gamma_0$ is not a uniform lattice in $R$.
    Since lattices in solvable groups are uniform (Theorem \ref{latts_in_sol}(i)),
    $\Gamma_0$ is not a lattice in $R$.
\end{proof}

\begin{rmk}
    The discussion after \cite[I.4 Thm.~1.6]{onishchik2} includes the remark
    that $R$ admits no lattices. This appears to be a mistake, since $R$ is
    a product of $\R^2 \rtimes \op{SO}(1,1)^0$ and $\R^3$,
    both of which admit lattices.
    Explicitly, the above construction produces a lattice of $G$
    lying in $R$ when given
    \begin{align*}
        P_s(x) &= x^2 - 3x + 1 \\
        r &= \op{id}_{\R^3} .
    \end{align*}
\end{rmk}

\section{Proofs of Theorem \ref{main} and Corollary \ref{cor:main}} \label{sec:proof}

The proof given in this section follows the same general method as both
Wu's in \cite[1.3]{wu} and Mostow's in \cite[Lemma 3.9]{mostow}.
We repair the step made explicit by Wu and also
use it to prove Corollary \ref{cor:main}.

\subsection{A key Lemma}

Mostow and Wu prove the following Lemma using algebraic groups;
Wu appears to use a decomposition like the one in \cite[Thm.~1.5.6]{onishchik3}.
At the risk of causing
further confusion, we give yet another proof, using Lie algebras
and hiding the use of algebraic groups behind Chevalley's Theorem
(Thm.~\ref{thm:chevalley}).
\begin{lemma}[{\cite[Lemma 3.8]{mostow}}, see also {\cite[Lemma 1.1]{wu}}] \label{nilradical_is_maximal}
    Let $G$ be a connected Lie group whose semi\-simple part $S$ is compact.
    If $S$ contains no nontrivial connected closed subgroup
    that is normal in $G$, then the nilradical $N$ of $G$
    is a maximal connected nilpotent subgroup.
\end{lemma}
\begin{proof}
    Let $\lie{n}$, $\lie{r}$, $\lie{g}$, and $\lie{s}$ be the Lie algebras
    of $N$, $R$ (the radical), $G$, and $S$. Suppose $N_1 \supsetneq N$ is a connected
    nilpotent subgroup of $G$ with Lie algebra $\lie{n}_1$. We show that if $S$ is
    compact then $\lie{s}$ contains a nonzero ideal of $\lie{g}$.

    Taking coordinates from the Levi decomposition
    $\lie{g} = \lie{r} + \lie{s}$ (a direct sum of vector spaces),
    pick $r+s \in \lie{n}_1 \smallsetminus \lie{n}$.
    Since $S$ is compact, we may make $r$ invariant by averaging.
    That is, if $\mu$ is normalized Haar measure on $S$, then
        \[ r'= \int_{g \in S} \op{Ad}_g (r) \,d\mu(g) \]
    is $S$-invariant.
    Since $S$ acts trivially on $\lie{r}/\lie{n}$ due to Chevalley's Theorem (Thm.~\ref{thm:chevalley}),
    $r' - r \in \lie{n}$. Thus $r' + s \in \lie{n}_1 \smallsetminus \lie{n}$.

    Let $\op{ad}_{\lie{n}}$ denote the adjoint action of $\lie{g}$ on $\lie{n}$.
    Since $\lie{n}_1$ is nilpotent and $r'$ is $S$-invariant,
    the following is a Jordan-Chevalley decomposition.
        \[ \op{ad}_{\lie{n}} r' = \op{ad}_{\lie{n}} (r'+s) - \op{ad}_{\lie{n}} s \]
    Acting by any element of $S$ fixes $r'$ and replaces $s$ with some $s'$.
    By uniqueness of the Jordan-Chevalley decomposition,
    $\op{ad}_{\lie{n}}$ is zero on the $S$-orbit $S(s-s') \subset \lie{s}$.
    The subspace this generates is an ideal $\lie{a}$ of $\lie{s}$ with trivial action on $\lie{n}$
    and trivial action on $\lie{r}/\lie{n}$---thus an ideal of $\lie{g}$.

    Since $\lie{r}/\lie{n}$ is abelian, any subalgebra of $\lie{r}$
    containing $\lie{n}$ is an ideal of $\lie{r}$. So $\lie{n}_1$,
    being nilpotent and properly containing the nilradical $\lie{n}$, cannot lie in $\lie{r}$.
    Thus we may assume $s \neq 0$. Then since $S$ is semisimple,
    we can take $s' \neq s$, which makes $\lie{a}$ nonzero.
\end{proof}

\subsection{Proof of Theorem \ref{main}}

\begin{notation}
    If $A$ is a subset of a topological group $G$,
    then $\overline{A}$ denotes its closure and $A^0$
    denotes its identity component.
\end{notation}

\begin{proof}[Proof of Theorem \ref{main}]
Let $G$ be a connected Lie group with solvable part $R$,
nilradical $N$, and semisimple part $S$.
Assume no nontrivial compact factor of $S$ acts trivially on $R$.
Given a lattice $\Gamma$ in $G$, let $R_1 = \overline{\Gamma R}^0$.

\paragraph{Step 1: $R_1$ is solvable.} \label{step:auslander}

Since $R$ is normal in $G$, the set $\Gamma R$ is a subgroup of $G$.
Then $R_1$ and $\Gamma R_1 = \Gamma(\overline{\Gamma R}^0) = \overline{\Gamma R}$
are both closed subgroups.

Solvability of $R_1$ will follow from this theorem of Auslander.
\begin{thm}[{\cite[Prop.~2]{auslander_radicals_1963}}; see also {\cite[Thm.~8.24]{raghunathan}}] \label{thm:solvable_closure}
    In a Lie group $G$, let $R$ be a closed, connected, simply connected
    normal solvable subgroup
    and let $\Gamma$ be a discrete subgroup. Then $\overline{\Gamma R}^0$ is solvable.
\end{thm}

In our situation, $R$ is not simply connected. However,
Mostow notes in \cite[2.6.1]{mostow} that the conclusion still holds.
One can use the version in \cite[Thm.~8.24]{raghunathan} or derive it from the original as follows.

Let $\pi: \tilde{G} \to G$ be the universal cover of $G$.
The Levi decomposition of $\tilde{G}$ splits \cite[\S{}1.4.1]{onishchik3},
so the inclusion $R \hookrightarrow G$ lifts to an injection
on universal covers $\tilde{R} \to \tilde{G}$.
Multiplication by $\ker \pi$ preserves $\pi^{-1}(\Gamma R)$,
so $\pi$ restricts to a covering map
$\overline{\pi^{-1}(\Gamma) \tilde{R}} = \overline{\pi^{-1}(\Gamma R)} \to \overline{\Gamma R}$.
Then $\overline{\Gamma R}^0$ is covered by
$\overline{\pi^{-1}(\Gamma) \tilde{R}}^0$, which is solvable by
Theorem \ref{thm:solvable_closure}.

\paragraph{Step 2: The nilradical of $R_1$ is the nilradical of $G$.} \label{step:nilradical}

Using Borel's density theorem, Mostow proves the following.
\begin{lemma}[{\cite[Lemma 3.4(d)]{mostow}}]
    Let $R$ be the radical of a connected Lie group $G$.
    If $\Gamma$ is a closed subgroup of $G$ such that $\Gamma^0 \subseteq R$
    and $\Gamma \backslash G$ has finite volume,
    then $\overline{\Gamma R}^0 \subseteq RK$
    where $K$ is a maximal compact factor of the semisimple part $S$.
\end{lemma}

Normal subgroups of $G$
lying in $S$ must commute with $R$ (since their tangent algebras are ideals
in the Lie algebra of $G$), so the hypothesis of
Lemma \ref{nilradical_is_maximal} is satisfied when
no compact factor of $S$ acts trivially on $R$.
Then $N$---which is the nilradical of $G$ and thus also that of $R$ and $RK$---is
a maximal connected nilpotent subgroup of $RK$
by Lemma \ref{nilradical_is_maximal}.
Since $R_1 \subseteq RK$, maximality makes $N$ the nilradical of $R_1$.

\paragraph{Step 3: $N$ is $\Gamma$-hereditary (Part (i)).}

This theorem of Mostow,
applied to the subgroups $\Gamma R_1 \supseteq \Gamma$, implies
$\Gamma R_1 / \Gamma = R_1 / (\Gamma \cap R_1)$ has finite volume.
\begin{thm}[{\cite[Lemma 2.5]{mostow_homogeneous_1962}}] \label{lemma:nested}
    Let $G$ be a locally compact topological group and let $F \supseteq E$
    be closed subgroups. If $G/E$ has a finite invariant measure $m$, then $G/F$
    and $F/E$ admit finite invariant measures of which $m$ is a product.
\end{thm}
Therefore $R_1$ is $\Gamma$-hereditary.
As the nilradical in a solvable group,
$N$ is lattice-hereditary in $R_1$ (Theorem \ref{latts_in_sol}(ii)).
Thus $\Gamma \cap R_1 \cap N = \Gamma \cap N$
is a lattice in $N$.
Since $\Gamma$ is an arbitrary lattice of $G$, this proves (i).

\paragraph{Step 4: Wu's reduction of Part (ii) to Part (i).} \label{step:last}

In this part, we assume additionally that $S$ acts on $R/N$ without
compact factors in the kernel. Since
$R/N$ is both the radical and the nilradical of $G/N$, it
is lattice-hereditary in $G/N$ by Part (i).

By Theorem \ref{heredity_and_closure}, heredity is equivalent
to having the quotient map take lattices to lattices.
So $G \to G/N \to (G/N)/(R/N) = G/R$ sends $\Gamma$ to a lattice.
Thus $R$ is lattice-hereditary in $G$.
\end{proof}

\begin{rmk} \label{rmk:what_is_broken}
    It is at the last step (step 4) above that Wu's proof in
    \cite[1.3]{wu} omits the condition involving $R/N$.
    Raghunathan's proof, in \cite[8.28]{raghunathan},
    stops after obtaining $\Gamma$-heredity of $N$
    and leaves readers to reconstruct the rest.
    (See, however, \cite[\S{ }2]{wu} where Wu discusses
    an earlier problem in Raghunathan's proof.)

    To spell it out, the problem is this: although the action
    of $S$ on $R$ might have no compact factors in the kernel,
    the same is not automatically guaranteed for the induced action
    of $S$ on $R/N$.

    For example: in Example \ref{eg:starkov},
    the nilradical $N$ of $G$ is $\R^2 \times \R^3$,
    and $N \cap \Gamma = \Gamma_0$ is indeed a lattice in $N$.
    Passing to $G/N$ yields Example \ref{eg:minimal}.
\end{rmk}

\begin{rmk} \label{rmk:chevalley}
    In view of Chevalley's Theorem (Thm.~\ref{thm:chevalley}),
    $G$ acts trivially on $R/N$.
    Thus $S$ has no compact factor acting trivially on $R/N$ if and only if
    $S$ has no compact factor.
    This simplification makes Part (ii) a case of
    Auslander's theorem \cite[Thm.~1]{auslander_radicals_1963}, which
    suggests the statement of Corollary \ref{cor:main}.
\end{rmk}

\subsection{Proof of Corollary \ref{cor:main}} \label{sec:proof_cor}

Let $G$ be a connected Lie group
with Levi decomposition $G = RS$ and nilradical $N$.
Let $C$ and $S_K$ be the maximal connected semisimple compact normal
subgroups of $G$ and $S$, respectively.
\begin{proof}[Proof of Corollary \ref{cor:main}]
    $C$ is compact and thus lattice-hereditary in $G$.
    It is normal by assumption and closed by compactness,
    so $G/C$ is a Lie group. We will pass to $G/C$ and continue this pattern.

    A normal subgroup of $S$ acting trivially on $R$ is normal in $G$,
    so $G/C$ satisfies Part (i) of Theorem \ref{main}.
    The nilradical of $G/C$ is $NC/C$, which is thus closed, normal,
    and lattice-hereditary.

    Since $S_K$ acts trivially on $R/N$ by Chevalley's Theorem (Thm.~\ref{thm:chevalley}) and is
    normal in $S$, its image $NS_K/(NC)$ in $G/(NC)$ is normal.
    Since $S_K$ is compact, $NS_K/(NC)$ is also closed and lattice-hereditary.

    $RS_K/(NS_K)$ is the nilradical of $G/(NS_K)$,
    whose semisimple part has no compact factors by the definition of $S_K$.
    So by Part (i) again, $RS_K / (NS_K)$ is lattice-hereditary in $G/(NS_K)$.

    Then by Theorem \ref{heredity_and_closure},
    a lattice in $G$ maps to a lattice under quotients by
    each of the subgroups $C \subseteq NC \subseteq NS_K \subseteq RS_K$.
    By the same theorem, each of these
    subgroups is lattice-hereditary in $G$.
\end{proof}

\section{Related results} \label{sec:others}

For a summary of other known results on heredity, see
\cite[\S{}2.1]{van_limbeek_riemannian_2014}.
Proofs can be found in \cite[\S{}I.1.4]{onishchik2}
and in \cite[Chapter 1]{raghunathan} starting with Theorem 1.12.

The version of Theorem \ref{main} in \cite[I.4 Thm.~1.6]{onishchik2} also
cites \cite{wolf}. A scan through the index
and through chapters with promising-looking section titles did not
reveal the location of this statement to the author.

When $G$ is a \emph{complex} Lie group, $S$ has no compact factors.
In this situation,
a shorter proof is possible, using
the Borel density theorem
(see e.g.\ \cite[I.1 Thm.~8.2]{onishchik2})
to show that $R_1 = R$.
See e.g.\ \cite[Thm.~3.5.3]{winkelmann}.

When $H$ is a connected, simply-connected solvable Lie group
and $K \subseteq \op{Aut} H$ is compact,
Dekimpe, Lee, and Raymond give conditions in \cite{dlr2001}
for $H$ to be lattice-hereditary in $H \rtimes K$.

Instead of studying $\Gamma/(\Gamma \cap R)$ in $G/R$,
one may take the quotient of $\Gamma$ by its maximal solvable normal subgroup.
In \cite[Lemma 6]{prasad}, Prasad relates this quotient to a lattice in
(a group covered by) $G/(RS_K)$.

\bibliographystyle{amsalpha}
\newcommand{\MR}[1]{\MRhref{#1}{MR #1}}
\providecommand{\MRhref}[2]{\href{http://www.ams.org/mathscinet-getitem?mr=#1}{#2}}
{\small \bibliography{lattices}}

\bigskip
\noindent
\textsc{Department of Mathematics} \\
\textsc{University of Chicago} \\
\textsc{Chicago, IL 60637} \\
\textit{E-mail address}: \url{ageng@uchicago.edu}

\end{document}